\documentclass[a4paper, 12pt]{amsproc}
\allowdisplaybreaks[1]

\usepackage{amsmath, amsthm, amssymb, mathtools, mathrsfs, stmaryrd}
\usepackage{ascmac}
\usepackage{comment}
\usepackage{bm}
\usepackage{ytableau}

\usepackage{hyperref}

\usepackage{tikz}
\usepackage{graphicx}
\usepackage{here}
\usepackage{time}
\usepackage[abbrev]{amsrefs}

\usepackage{xcolor}
\usepackage[capitalize,nameinlink,noabbrev,nosort]{cleveref}
\hypersetup{
 colorlinks=true,       % false: boxed links; true: colored links
 linkcolor=blue,          % color of internal links
 citecolor=blue,        % color of links to bibliography
 filecolor=blue,      % color of file links
 urlcolor=blue,           % color of external links
}

\usepackage{fullpage}

\makeatletter
\@namedef{subjclassname@2020}{%
  \textup{2020} Mathematics Subject Classification}
\makeatother

\newtheorem{theoremcounter}{Theorem Counter}[section]

\theoremstyle{definition}

\newtheorem{remark}[theoremcounter]{Remark}

\theoremstyle{plain}
\newtheorem{lemma}[theoremcounter]{Lemma}
\newtheorem{proposition}[theoremcounter]{Proposition}
\newtheorem{corollary}[theoremcounter]{Corollary}

\newtheorem{theorem}[theoremcounter]{Theorem}

\numberwithin{equation}{section}

\newcommand{\Z}{\mathbb{Z}}
\newcommand{\Q}{\mathbb{Q}}
\newcommand{\R}{\mathbb{R}}
\newcommand{\C}{\mathbb{C}}

\DeclareMathOperator{\ReNew}{Re}
\renewcommand{\Re}{\ReNew}

\begin{document}
\author{Takashi Miyagawa}
\address[Takashi Miyagawa]{Onomichi City University,  1600-2 Hisayamada-cho, Onomichi, Hiroshima, 722-8506, Japan} 
\email{miyagawa@onomichi-u.ac.jp}

%\author{Hideki Murahara}
%\address[Hideki Murahara]{The University of Kitakyushu,  4-2-1 Kitagata, Kokuraminami-ku, Kitakyushu, Fukuoka, 802-8577, Japan}
%\email{hmurahara@mathformula.page}

%%%%%%%%%%%%%%%%%%%%%%%%%%%%%%%%%%%%%%%%%%%%%%%%%%%%%%%%%%%%

%%%%%%%%%%%%%%%%%%%%%%%%%%%%%%%%%%%%%%%%%%%%%%%%%%%%%%%%%%%%
\begin{abstract}
Due to their deep connection with the Riemann zeta function, 
the asymptotic behavior of mean values of multiple zeta functions 
has attracted considerable attention.

In this paper, we study the mean square values of Hurwitz-type 
and Barnes-type multiple zeta functions. 
For the Hurwitz-type multiple zeta function, 
we establish asymptotic formulas and upper bounds for its 
mean square values in terms of the parameter $\sigma$.
Our approach relies on the fact that Hurwitz-type multiple 
zeta functions can be expressed as linear combinations of 
the classical Hurwitz zeta function, which allows us to apply 
known results on the mean values of the latter almost directly.
For the Barnes-type multiple zeta function, 
we show that the behavior of the mean square values depends 
essentially on the arithmetic structure of the parameter vector.
In the case where the parameters are linearly dependent over $\Q$, 
we obtain asymptotic formulas analogous to the Hurwitz-type case. 
In contrast, for general parameters, we derive upper bounds 
for the mean square values from bounds for the function itself.
In particular, we clarify how the order of the mean square values 
varies in terms of the dimension of the $\Q$-vector space spanned by 
the parameters of the Barnes zeta function.
\end{abstract}

%%%%%%%%%%%%%%%%%%%%%%%%%%%%%%%%%%%%%%%%%%%%%%%%%%%%%%%%%%%%

%%%%%%%%%%%%%%%%%%%%%%%%%%%%%%%%%%%%%%%%%%%%%%%%%%%%%%
\title{Mean values and upper bounds for the Hurwitz and Barnes multiple zeta functions}

\maketitle

%%%%%%%%%%%%%%%%%%%%%%%%%%%%%%%%%%%%%%%%%%%%%%%%%%%%%%
\section{Introduction and main theorems}

Let $r $ be a positive integer and let $s=\sigma+it \ (\sigma,t\in \R)$ be a complex variable. Let $a>0$ and $w_1,\dots w_r>0$. 
The Barnes multiple zeta function, denoted by $\zeta_r(s,a,(w_1, \dots,w_r))$ and introduced in \cites{Barnes1899, Barnes1901, Barnes1904}, is defined as  
\begin{align}\label{zeta_r}
 \zeta_r (s,a,(w_1,\dots,w_r))
 =\sum_{m_1=0}^\infty \cdots \sum_{m_r=0}^\infty
  \frac{1}{(a+m_1 w_1+\cdots+m_r w_r)^s}.
\end{align}
In particular, for $(w_1,\dots,w_r)=(1,\dots,1)$, the series
\begin{align}\label{zeta_r(1,...,1)}
  \zeta_r(s,a,(1,\dots,1)) = \sum_{m_1=0}^\infty \cdots \sum_{m_r=0}^\infty
  \frac{1}{(a+m_1 +\cdots+m_r)^s}
\end{align}
is sometimes called the Hurwitz multiple zeta function.
The series \eqref{zeta_r} and \eqref{zeta_r(1,...,1)} are natural generalizations of the Hurwitz zeta function
\begin{align}\label{zeta_H}
  \zeta_H(s,a) = \sum_{m=0}^\infty \frac{1}{(m+a)^s} \quad (\sigma>1).
\end{align}
Furthermore, \eqref{zeta_r} and \eqref{zeta_r(1,...,1)} converge absolutely for $\mathrm{Re}(s) > r$ and can be meromorphically continued throughout the entire complex $s$-plane. 
In addition, these functions have simple poles at $s=1,2,\dots,r$.

For convenience in this paper, we introduce the following boldface symbols to represent index tuples:
\begin{align*}
 &\mathbf{m}=(m_1,\dots,m_r), \quad
  \mathbf{n}=(n_{1},\dots,n_{r}), \quad\\
 &\mathbf{w}=(w_1,\dots,w_r), \ \, \quad
  \mathbf{1}=(1,\dots,1).
\end{align*}
We assume that $w_1,\dots,w_r\in \mathbb{R}_{>0}$ throughout this paper.
Using this notation, \eqref{zeta_r} and \eqref{zeta_r(1,...,1)} can be rewritten as  
\begin{align}
 \zeta_r(s,a,\mathbf{w})
 =\sum_{m_1,\dots,m_r\ge 0}
 \frac{1}{(a+\mathbf{m}\cdot\mathbf{w})^s}
\end{align}
and
\begin{align}
 \zeta_r(s,a,\mathbf{1})
 =\sum_{m_1,\dots,m_r\ge 0}
 \frac{1}{(a+\mathbf{m}\cdot\mathbf{1})^s}, \label{HMZ}
\end{align}
respectively.
In particular, $\zeta_r(s,a,\mathbf{1})$ is referred to as the Hurwitz multiple zeta function, and it can be expressed as a linear combination of single Hurwitz zeta functions (see \cite{SrivastavaChoi2001} p. 86):  
\begin{align}  
  \zeta_r(s,a,\mathbf{1})=\sum_{j=0}^{r-1}p_{r,j}(a)\zeta_H(s-j, a). 
  \label{HMZ-HZ}
\end{align}
In this formula, $p_{r,j}(a)$ is given by  
\begin{align*}  
 p_{r,j}(a)=\frac{1}{(r-1)!}\sum_{l=j}^{r-1}(-1)^{r+1-j}\binom{l}{j}S(r,l+1) a^{l-j},  
\end{align*}  
where $S(r,l+1)$ denotes the Stirling number of the first kind.
Combining the above formula with known results on the mean square values of the Hurwitz zeta function, we obtain the following theorem concerning the order of the mean square value: 
% of  $\int_1^T|\zeta_r(\sigma+it,a,\mathbf{1})|^2 dt$ as $T \rightarrow \infty$ 

\begin{theorem}\label{zeta_r(1/2,a)}
Let $a>0, s=\sigma+it$ and assume $r-1<\sigma < r$. Then the order of the mean square values of $\zeta_r(\sigma+it,a,\mathbf{1})$ is given in the following three cases:
\begin{enumerate}
    \item[(i)] If $r-1/2< \sigma<r$,
    \begin{align*}
        \int_1^T |\zeta_r(\sigma+it,a,\mathbf{1})|^2dt&=
        \sum_{0 \le k,l \le r-1}p_{r,k}(a)p_{r,l}(a) \zeta_H(2\sigma-k-l,a)T \\
        & \quad + \frac{1}{\{(r-1)!\}^2} \cdot \frac{(2\pi)^{2\sigma-2r+1}}{2r-2\sigma} \zeta(2r-2\sigma)T^{2r-2\sigma}+O(T^{r-\sigma}\log{T}),
    \end{align*}
    \item[(ii)] if $\sigma=r-1/2$,
\begin{align*}
 &\int_1^T \left|\zeta_r\left(r-\frac{1}{2}+it,a,\mathbf{1}\right)\right|^2 dt \\
 & \quad = \frac{1}{\{(r-1)!\}^2} T \log{T} \\
 & \quad \qquad +\left\{
 \sum_{(k,l)\ne(r-1,r-1)} p_{r,k}(a)p_{r,l}(a) \zeta_H(2\sigma-k-l,a)+
 \frac{\gamma(a)+\gamma-1-\log{(2\pi)}}{((r-1)!)^2}\right\}T \\
 & \quad \qquad  +O\big(T^{1/2}\log{T}\big),
\end{align*}
    \item[(iii)]
    if $r-1<\sigma<r-1/2$, 
    \begin{align*}
    \int_1^T|\zeta_r(\sigma+it,a,\mathbf{1})|^2 dt 
    & =\frac{1}{\{(r-1)!\}^2}\frac{(2\pi)^{2\sigma-2r+1}}{2r-2\sigma}\zeta(2r-2\sigma)T^{2r-2\sigma} \\
    & \quad +\sum_{0 \le k,l \le r-1}p_{r,k}(a)p_{r,l}(a) \zeta_H(2\sigma-k-l,a)T +O(T^{r-\sigma}\log{T}).
    %+O\big(T^{2r-2\sigma-1/2}\big).
\end{align*}
\end{enumerate}
as $ T \to \infty $. Here $\gamma$ denotes Euler's constant, and $\gamma(a)$ the generalized Euler constant, defined by
\[
    \gamma=\lim_{M \rightarrow \infty} \left( \sum_{m=1}^M \frac1{m} -\log{M} \right), \quad \gamma(a) = \lim_{M \rightarrow \infty} \left\{ \sum_{m=0}^M \frac1{m+a} -\log(M+a) \right\}.
\]
\end{theorem}
\bigskip

 \begin{remark}
     We remark that some of the diagonal contributions in case (iii) of the above theorem involve values of the Hurwitz zeta function at points outside the Dirichlet series region $\Re (s)>1$.
 More precisely, for the pair $k=l=r-1$ we have $2\sigma-(k+l)=2\sigma-2r+2\in(0,1)$, so the series $\sum_{m\ge0}(m+a)^{-2\sigma+2r-2}$ diverges.
 Equivalently, one may obtain the same contribution by using the standard finite sum approximation \eqref{Hurwitz_finite}
     $$\zeta_H(s,a)=\sum_{m\le x}(m+a)^{-s}+\frac{x^{1-s}}{s-1}+O(x^{-\sigma}),$$
 with $x\asymp T$
 , and then exchange the order of summation and integration.
 This procedure yields the $T$-term $\zeta_H(2\sigma-2r+2,a)T$, together with the secondary power term 
         $$\dfrac{(2\pi)^{2\sigma-2r+1}}{2r-2\sigma}\zeta(2r-2\sigma)T^{2r-2\sigma}$$ 
    as appears in the text.
 \end{remark}

\bigskip

On the other hand, the mean square value of the Barnes multiple zeta function
$\zeta_r(s,a,\mathbf{w})$ depends sensitively on the arithmetic structure of the
parameter vector $\mathbf{w}=(w_1,\dots,w_r)$.
In \cite{MiyagawaMurahara2025}, it was shown in the case $d=1$ that
$
\int_1^T \zeta_r(\sigma+it,a,\mathbf{w})\,dt
$
is given by a linear function of $T$ in the range $\sigma>r-1/2$.
In the case where the components of $\mathbf{w}$ are $\mathbb{Q}$-linearly dependent
(i.e.\ $d=1$), the values of $\mathbf{m}\cdot\mathbf{w}$ exhibit strong multiplicities,
and the problem can be reduced to the Hurwitz-type case, leading to precise asymptotic formulas.

In the present paper, we extend this analysis to the range $r-1<\sigma\le r-1/2$.
In contrast, when $2 \le d \le r$, such a reduction is no longer available.
The main difficulty in the Barnes-type case lies in the treatment of
off-diagonal terms, which cannot be controlled by standard methods
when $d\ge 2$.
In particular, the off-diagonal terms become much more difficult to control,
and at present we restrict ourselves to upper bounds obtained from
general estimates for $\zeta_r(s,a,\mathbf{w})$.
Thus, while the $\mathbb{Q}$-rank
$d=\dim_{\mathbb{Q}}\langle w_1,\dots,w_r\rangle$
plays a fundamental role, our results provide a complete description only in the
case $d=1$, and more limited estimates in the higher-rank cases.
This leads to the following theorem.

\bigskip

% \red{
% \begin{theorem}\label{main4_revised_v3}
% Let $0 < a \le 1$, $s = \sigma + it$, and
% $d = \dim_{\mathbb{Q}} \langle w_1, \dots, w_r \rangle $.
% Then the mean square values of $\zeta_r(\sigma+it, a, \mathbf{w})$ satisfy the following estimates.
% \begin{enumerate}
%     \item[(i)] If $d=1$, then there exist $\lambda > 0$ and integers $p_1, \dots, p_r \in \mathbb{Z}_{>0}$ such that $w_j = \lambda p_j \ (1 \le j \le r)$ with $\gcd(p_1, \dots, p_r) = 1$.
%     For $\sigma = r-1/2$, we have
%     \begin{align*}
%         \int_1^T \left|\zeta_r\left(r-\frac{1}{2}+it, a, \mathbf{w}\right)\right|^2 dt = \frac{\lambda^{1-2r}}{\{(r-1)!\}^2 (p_1 \cdots p_r)^2} T \log T + O(T).
%     \end{align*}
%     For $r-1 < \sigma < r-1/2$, we have
%     \begin{align*}
%         \int_1^T |\zeta_r(\sigma+it, a, \mathbf{w})|^2 dt 
%         %= \frac{\lambda^{1-2r}}{\{(r-1)!\}^2 (p_1 \cdots p_r)^2 (2r-2\sigma-1)} T^{2r-2\sigma} + o(T^{2r-2\sigma}).
%         \asymp T^{2r-2\sigma}.
%     \end{align*}
%     \item[(ii)] If $2\le d\le r$ and the associated vector $\mathbf v=(v_1,\dots,v_d)$
% may be chosen so that $1,v_1,\dots,v_d$ are linearly independent algebraic numbers
% over $\mathbb Q$, then for any $\varepsilon>0$ and any $\sigma>r-1$, we have
% \[
% \int_1^T |\zeta_r(\sigma+it,a,\mathbf w)|^2dt
% \ll T^{2r-2\sigma+d+\varepsilon}.
% \]
% \end{enumerate}
% In all cases, the implied constants may depend on $r,\sigma,a,\mathbf{w}$ and $\varepsilon$.
% \end{theorem}
% }

\begin{theorem}\label{main4_revised_v3}
Let $a>0, s = \sigma + it$ and
$d = \dim_{\mathbb{Q}} \langle w_1, \dots, w_r \rangle$.
Then the mean square values of $\zeta_r(\sigma+it, a, \mathbf{w})$ satisfy the following.
\begin{enumerate}
\item[(I)] The case $d=1$. There exist $\lambda > 0$ and integers
$p_1, \dots, p_r \in \mathbb{Z}_{>0}$ such that
$w_j = \lambda p_j \ (1 \le j \le r)$ with
$\gcd(p_1, \dots, p_r) = 1$.
\begin{enumerate}
\item[(i)] If $\sigma = r-1/2$, then
\[
\int_1^T \left|\zeta_r\left(r-\frac{1}{2}+it, a, \mathbf{w}\right)\right|^2 dt
=
\frac{\lambda^{1-2r}}{\{(r-1)!\}^2 (p_1 \cdots p_r)^2}
\, T \log T
+ O(T).
\]
\item[(ii)] If $r-1 < \sigma < r-1/2$, then
\[
\int_1^T |\zeta_r(\sigma+it, a, \mathbf{w})|^2 dt
\asymp T^{2r-2\sigma}.
\]
\end{enumerate}
\item[(II)] The case $2 \le d \le r$.
For any $\mathbf w \in \mathbb{R}_{>0}^r$ with $2 \le d \le r$, we have the following:
\[
\int_1^T |\zeta_r(\sigma+it,a,\mathbf w)|^2dt \ll
    \begin{cases}
    T & \text{if } \sigma>r,\\[6pt]
    T (\log T)^2 & \text{if } \sigma= r,\\[6pt]
    T^{2r-2\sigma+1} & \text{if } r-1<\sigma<r.
    \end{cases}
\]
\end{enumerate}
In all cases, the implied constants may depend on
$r,\sigma,a,$ and $\mathbf{w}$.
\end{theorem}
In particular, these bounds show that, in contrast to the rank-one case,
no power saving beyond the trivial bound is currently available in higher rank.
This demonstrates a sharp contrast with the case $d=1$,
where precise asymptotics are available.

%If $2\le d\le r$ and the associated vector $\mathbf v$ consists of algebraic numbers satisfying $1,v_1,\dots,v_d$ linearly independent over $\mathbb Q$, then ...

\bigskip

\begin{remark}\label{remark:Caw_residue}
Assume that $d=1$ and write $w_j=\lambda p_j$ as in Theorem ~\ref{main4_revised_v3}.
Let
$
A(n)
=
\#\left\{\mathbf m\in\mathbb Z_{\ge0}^r:\ p_1m_1+\cdots+p_rm_r=n\right\}.
$
Then the diagonal series can be written as
\[
\tilde{\zeta}_r(\sigma,a,\mathbf w)
=
\sum_{\substack{m_1,\dots,m_r\ge 0 \\ n_1,\dots,n_r\ge 0 \\ \mathbf m \cdot \mathbf{w} = \mathbf n \cdot \mathbf{w}}}\frac{1}{(a+\mathbf{m}\cdot\mathbf{w})^{\sigma}(a+\mathbf{n}\cdot\mathbf{w})^{\sigma}}
=
\sum_{n\ge0}\frac{A(n)^2}{(a+\lambda n)^{2\sigma}}.
\]
Formally, the constant $C(a,\mathbf w)$ corresponds to the logarithmic divergence
\[
C(a,\mathbf w)
=
\lim_{X\to\infty}\frac{1}{\log X}
\sum_{n\le X}\frac{A(n)^2}{(a+\lambda n)^{2r-1}},
\]
whenever the limit exists. Moreover, if $\tilde{\zeta}_r(\sigma,a,\mathbf w)$ admits
a meromorphic continuation to a neighborhood of $\sigma=r-1/2$ with a simple pole there,
then $C(a,\mathbf w)$ equals the residue of $\tilde{\zeta}_r(\sigma,a,\mathbf w)$ at $\sigma=r-1/2$. 
\end{remark}

\bigskip

To establish the upper bounds in the higher-rank case $2 \le d \le r$,
we will later use general growth estimates for $\zeta_r(s,a,\mathbf{w})$ (in Theorem\ref{thm:Barnes_bounds}),
which are stated after the main theorem.

\bigskip

\begin{theorem}\label{thm:Barnes_bounds}
Let $s=\sigma+it$ with $t\ge 2$, $a>0$, and
$\mathbf{w}\in\mathbb{R}_{>0}^r$.
Then the following bounds hold:
\[
\zeta_r(\sigma+it,a,\mathbf{w}) \ll 
\begin{cases}
1 
& \text{if } \quad   \sigma> r,\\[6pt]
\log t
& \text{if } \quad  \sigma=r,\\[6pt]
t^{\,r-\sigma}
& \text{if } \quad  r-1< \sigma < r,
\end{cases}
\]
as $t \to \infty$, uniformly for $\sigma$ in any fixed compact subinterval
of each region.
The implied constants depend on $a$, $r$, $\mathbf{w}$, and the chosen strip.
\end{theorem}

\bigskip

\section{Preliminaries on the Hurwitz and Lerch zeta functions}
 Next, we state some fundamental facts about estimation of the Hurwitz zeta function.
For $0<\sigma_0\le \sigma \le 2$, $x\ge |t|/\pi$ and $s=\sigma+it$, we have
\begin{align}\label{Hurwitz_finite}
\zeta_H(s,a)=\sum_{m\le x}\frac{1}{(m+a)^s}+\frac{x^{1-s}}{s-1}+O(x^{-\sigma}),
\end{align}
uniformly as $|t|\to\infty$. Then we have
\begin{align*}
    & \zeta_H(1+it, a) = \sum_{m\le t}\frac{1}{(m+a)^{1+it}}+O(t^{-1}) \ll \sum_{m \le t}\frac1{m+a} \ll \log{t}
    \quad (t \ge 2).
\end{align*} 
%The following bound follows from the functional equation of the Hurwitz zeta function together with the Phragmén--Lindelöf convexity principle; see \cite{Apostol1976,Titchmarsh1986}.
%Also, by applying the Phragmén–Lindelöf convexity principle, we obtain the estimate
Also, by applying the Phragmén--Lindelöf convexity principle
in the same way as for the Riemann zeta function
(see, for example, \cite{Ivic1985,Titchmarsh1986}),
we obtain the estimate
\begin{align}\label{Hurwitz_estimate}
    \zeta_H(\sigma+it,a) \ll 
    \begin{cases}
        1 & \text{if } \quad \sigma> 1,\\
        t^{(1-\sigma)/2} \log{t} & \text{if } \quad 0 \le \sigma \le 1,\\
        t^{1/2-\sigma} \log{t} & \text{if }   \quad \sigma \le 0,
    \end{cases}
\end{align}
as $ |t| \rightarrow \infty $. 
By combining \eqref{HMZ-HZ} and \eqref{Hurwitz_estimate}, we obtain following proposition.
\begin{proposition}\label{thm:HMZ_growth}
Let $s=\sigma+it$ with $t\ge 2$ and let $a>0$ be fixed. Then we have the following uniform bounds
\[
 \zeta_r(\sigma+it,a,\mathbf{1}) \ll 
 \begin{cases}
 1 & \text{if } \quad  \sigma> r,\\[4pt]
 t^{(r-\sigma)/2}\log t & \text{if } \quad  r-1< \sigma \le r,\\[4pt]
 t^{\,r-\sigma-1/2}\log t & \text{if } \quad  \sigma\le r-1,
 \end{cases}
 \]
 as $t \to \infty$. The implied constants depend on $a$ and $r$.
\end{proposition}

Let $0<a \le 1$ and $ \lambda \in \R $. The function
\[
  \zeta_L(s,a,\lambda) = \sum_{m=0}^\infty \frac{e^{2\pi i m \lambda }}{(m+a)^s} \quad (\sigma>1)
\]
is called the Lerch zeta function. 
% More generally, we use the notation
% \[
% \zeta_L(s,\alpha,\beta):=\sum_{m=0}^{\infty}\frac{e^{2\pi i m\alpha}}{(m+\beta)^s}
% \qquad (\sigma>1),
% \]
% so that $\zeta_L(s,a,\lambda)$ in the above sense is $\zeta_L(s,\lambda,a)$ in this notation.
When $\lambda=1$, it coincides with the Hurwitz zeta function $\zeta_H(s,a)$, which satisfies the functional equation
\begin{align}
 \zeta_H(1-s,a)
 &= 
 \frac{\Gamma(s)}{(2\pi)^{s}}
 \left\{ e^{-\pi is/2}\sum_{m=1}^\infty \frac{e^{2\pi ima}}{m^s}
 +e^{\pi is/2}  \sum_{m=1}^\infty \frac{e^{-2\pi ima}}{m^s} \right\}    \nonumber\\
 &=
 \Gamma(s) (2\pi)^{-s}
 \left\{ e^{-\pi is/2} \zeta_L(s,1,a)+ e^{\pi is/2} \zeta_L(s,1,1-a) \right\} \label{fq}
\end{align}
%for $ 0<a<1$, where $\zeta_L(s,1,a)=\sum_{m=1}{e^{2\pi ima}}/{m^s}$.
The functional equation \eqref{fq} is a meromorphic continuation for the whole $\mathbb{C}$.
Moreover, the Lerch zeta function $\zeta_L(s,a,\lambda)$ satisfies the functional equation
\begin{align}
\zeta_L(1-s,a,\lambda) 
&= \Gamma(s)(2\pi)^{-s} 
 \big\{  
  e^{\pi is/2 - 2\pi ia \lambda }\zeta_L(s,\lambda,1-a) \nonumber\\
& \qquad \qquad \qquad \qquad + e^{-\pi is/2 + 2\pi ia (1-\lambda) }\zeta_L(s,1-\lambda,a)
 \big\}
 \label{fq_Lerch}
\end{align}
for $0<\lambda<1$ and $0<a<1$.
\begin{proposition}[Theorem 2.1 and 2.2 for Chapter 4 in \cite{LaurincikasGarunkstis2002}]\label{zeta_L_H_1/2_1}
Let $ a>0,\  0<\lambda\le1 $. Then for $ s= \sigma + it \in \C $ with $1/2 < \sigma <1 $, we have
\begin{align}
  \int_1^T |\zeta_L(\sigma+it,a,\lambda)|^2 dt & = \zeta_H(2\sigma,a)T+\frac{(2\pi)^{2\sigma-1}}{2-2\sigma} \zeta_H(2-2\sigma,\lambda)T^{2-2\sigma}  \nonumber\\
  & \qquad  +O(T^{1-\sigma} \log{T}) + O(T^{\sigma/2}),  \label{G_and_L}
\end{align}
and
\begin{align*}
 \int_1^T \left|\zeta_L\left(\frac{1}{2}+it,a,\lambda\right) \right|^2 \,dt 
  = T\log{T} + \left(\gamma(a) + \gamma(\lambda) -1-\log{2\pi} \right)T + O(T^{1/2} \log{T})
\end{align*}
as $T \rightarrow \infty$.
\end{proposition}

\begin{remark}%[On the term $O(T^{\sigma/2})$ in \cite{LaurincikasGarunkstis2002}]
In \cite[Theorem~1]{GarunkstisLaurincikasSteuding2003b} and \cite[Theorem~2.1, p.68]{LaurincikasGarunkstis2002}, the error term $O(T^{\sigma/2})$ for \eqref{G_and_L} comes from an application of the Cauchy--Schwarz inequality to an
oscillatory integral of the form
\begin{align*}%\label{eq:LG_osc_term}
I(T)
& :=\int_{1}^{T} t^{\sigma/2-1}\sum_{0\le m\le k(t)}
\frac{e^{2\pi i\lambda m}}{(m+a)^{\sigma+it}}\,dt \\
& \ll \left( \int_{1}^{T} t^{\sigma-2} dt \right)^{1/2}
 \left( \int_{1}^{T} \left| \sum_{0\le m\le k(t)}
\frac{e^{2\pi i\lambda m}}{(m+a)^{\sigma+it}} \right|^2 dt \right)^{1/2} \ll T^{\sigma/2}
\end{align*}
where $k(t)=\lfloor (t/2\pi)^{1/2}-a \rfloor$. 
But the fact that the error term $O(T^{\sigma/2})$ can be improved is already pointed out by the authors themselves in the note on \cite[Note, p.69]{LaurincikasGarunkstis2002}.
They state as follows:
“Note that the error term 
$O(T^{\sigma/2})$ in Theorem~2.1 is not desirable. Perhaps it is possible to remove it in view of the results of \cite[Ivić and Matsumoto (1996)]{IvicMatsumono1996}.”
For the application in the present paper, it suffices to note that this contribution
is, in fact, bounded by $O(1)$ as $T\to\infty$. That is
\begin{align}
 \int_1^T |\zeta_L(\sigma+it,a,\lambda)|^2 dt = \zeta_H(2\sigma,a)T +\frac{(2\pi)^{2\sigma-1}}{2-2\sigma} 
 & \zeta_H(2-2\sigma,\lambda)T^{2-2\sigma} \nonumber \\
 & \qquad +O(T^{1-\sigma} \log{T}). \label{G_and_L_2}
\end{align}

\smallskip
\noindent\textbf{Sketch of proof.}
%Let $ k(t)= \lfloor (t/2\pi)^{1/2}-a \rfloor $.
Since $m\le k(t)$ is equivalent to $t\ge 2\pi(m+a)^2$, we may interchange the order of
summation and integration to obtain
\[
I(T)=\sum_{0\le m\le k(T)} \frac{e^{2\pi i\lambda m}}{(m+a)^{\sigma}}
\int_{2\pi(m+a)^2}^{T} t^{\sigma/2-1}e^{-it\log(m+a)}\,dt.
\]
%Put $\ell_m:=\log(m+a)>0$ and $t_0:=2\pi(m+a)^2$.
By integration by parts,
\begin{align*}
&\int_{2\pi(m+a)^2}^{T} t^{\sigma/2-1}e^{-it\log(m+a)}\,dt \\
&=
\Bigl[\frac{t^{\sigma/2-1}}{-i\log(m+a)}e^{-it\log(m+a)}\Bigr]_{2\pi(m+a)^2}^{T}
+\frac{\sigma/2-1}{i\log(m+a)}\int_{2\pi(m+a)^2}^{T} t^{\sigma/2-2}e^{-it\log(m+a)}\,dt,
\end{align*}
hence, using $\sigma/2-2<-1$,
\begin{align*}
&\int_{2\pi(m+a)^2}^{T} t^{\sigma/2-1}e^{-it\log(m+a)}dt \\
& \quad \ll
\frac{T^{\sigma/2-1}}{\log(m+a)}+\frac{(2\pi)^{\sigma/2-1}(m+a)^{\sigma-2}}{\log(m+a)}
+\frac{1}{\log(m+a)}\int_{2\pi(m+a)^2}^{\infty} t^{\sigma/2-2}\,dt \\
& \quad \ll
\frac{T^{\sigma/2-1}}{\log(m+a)}+\frac{(2\pi)^{\sigma/2-1}(m+a)^{\sigma-2}}{\log(m+a)}.
\end{align*}
Therefore
\begin{align*}
 |I(T)|
&\ll
T^{\sigma/2-1}\sum_{0 \le m\le k(T)}\frac{1}{(m+a)^{\sigma}\log(m+a)}
+\sum_{m\le k(T)}\frac{(2\pi)^{\sigma/2-1}(m+a)^{\sigma-2}}{(m+a)^{\sigma}\log(m+a)} \\
&\ll
T^{\sigma/2-1}\sum_{1 \le m\le k(T)}\frac{1}{m^{\sigma}\log m}+\sum_{m\ge 1}\frac{1}{m^{2}\log m}\\
&\ll
T^{\sigma/2-1}\cdot \frac{(k(T))^{1-\sigma}}{\log{k(T)}}.
\end{align*}
Since $k(T)\asymp T^{1/2}$, %and $(2\pi)^{\sigma/2-1}(m+a)^{\sigma-2}\asymp (m+a)^{\sigma-2}$, 
we have
\[
 |I(T)|\ll \frac{1}{T^{1/2}\log{T}} = o(1) \quad (T \rightarrow \infty).
\]
\qed
\end{remark}

\begin{proposition}\label{Hurwitz_0<siga<1/2}%\label{zeta_L_H_1/2}
For $ s= \sigma + it \in \C $ with $0 < \sigma <1/2 $. For $0<\lambda<1 $, we have
\begin{align*}
 \int_1^T|\zeta_L(\sigma+it,a,\lambda)|^2 dt
 =\frac{(2\pi)^{2\sigma-1}}{2-2\sigma}\zeta_H(2-2\sigma,1-\lambda)T^{2-2\sigma}+\zeta_H(2\sigma,a)T+O(T^{1-\sigma}\log{T}).
\end{align*}
Also in the case $\lambda=1$, we have
% \begin{align*}
%   \int_1^T |\zeta_H(s, a)|^2 dt
%   &= \frac{(2\pi)^{2\sigma-1}}{2-2\sigma} \zeta(2-2\sigma)T^{2-2\sigma} + O(T^{3/2 -2\sigma})
% \end{align*}
\begin{align*}
\int_1^T |\zeta_H(\sigma+it,a)|^2 dt 
    =\frac{(2\pi)^{2\sigma-1}}{2-2\sigma} \zeta(2-2\sigma) T^{2-2\sigma} +\zeta_H(2\sigma,a)T+O(T^{1-\sigma}\log{T})
\end{align*}
where $0<a<1$.
%Moreover, in both of the above formulas, the term $\zeta_H(2\sigma,a)T$ is absorbed into $O(T^{3/2-5\sigma/2})$ when $0<\sigma\le 1/5$.
\end{proposition}

\begin{proof}
We also recall Stirling’s formula for the gamma function:
\begin{align*}
    \Gamma(\sigma+it)
    =\sqrt{2\pi}\, t^{\sigma-1/2}e^{-\pi t/2}\bigl(1+O(1/t)\bigr)
    \qquad (t\to\infty).
\end{align*}
Suppose that $0 < \sigma_0 \le \sigma <1$. Taking the absolute square of both sides of \eqref{fq} and applying Stirling’s formula, we obtain
\begin{align}
 |\zeta_H(1-s,a)|^2
 &= (2\pi)^{1-2\sigma} t^{2\sigma-1} |\zeta_L(s,1,a)|^2
    + e^{-2\pi t} |\zeta_L(s,1,1-a)|^2 \nonumber\\
 &\quad
    + 2(2\pi)^{1-2\sigma} t^{2\sigma-1} e^{-\pi t}
      \mathrm{Re}\!\left(e^{-\pi i s}
      \zeta_L(s,1,a)\overline{\zeta_L(s,1,1-a)}\right) \nonumber\\
 &= (2\pi)^{1-2\sigma} t^{2\sigma-1}\bigl(1+O(1/t)\bigr)
    \times \Bigl\{
      |\zeta_L(s,1,a)|^2
      + e^{-2\pi t}|\zeta_L(s,1,1-a)|^2 \nonumber\\
 &\qquad\qquad\qquad
      + 2e^{-\pi t}\,
        \mathrm{Re}\!\left(
        e^{\pi i\sigma-2\pi i a}
        \zeta_L(s,1,1-a)\overline{\zeta_L(s,1,a)}
        \right)
    \Bigr\}
    \label{hz_fq}
\end{align}
as $t\to\infty$. A completely analogous computation applied to \eqref{fq_Lerch} yields
\begin{align}
 |\zeta_L(1-s,a,\lambda)|^2
 &= (2\pi)^{1-2\sigma} t^{2\sigma-1}\bigl(1+O(1/t)\bigr) \nonumber \\
 & \qquad
    \times \Bigl\{
      |\zeta_L(s,1-\lambda,a)|^2
      + e^{-2\pi t}|\zeta_L(s,\lambda,1-a)|^2 \nonumber\\
 &\qquad\qquad\qquad
      + 2e^{-\pi t}\,
        \mathrm{Re}\!\left(
        e^{\pi i\sigma-2\pi i a}
        \zeta_L(s,\lambda,1-a)\overline{\zeta_L(s,1-\lambda,a)}
        \right)
    \Bigr\},
 \label{sq_fq}
\end{align}
as $t\to\infty$. Suppose that $1/2<\sigma<1$.
Integrating \eqref{sq_fq} over $t\in[1,T]$, we obtain
\begin{align*}
 &\int_1^T|\zeta_L(1-s,a,\lambda)|^2 dt \\
 &= (2\pi)^{1-2\sigma} \bigl(1+O(1/t)\bigr)
  \Bigg\{ \int_1^T t^{2\sigma-1}|\zeta_L(\sigma+it,1-\lambda,a)|^2 dt + O(t^{2\sigma-1}e^{-2\pi T}) \\
 & \qquad \quad +\int_1^T2 t^{2\sigma-1}e^{-\pi t} \cdot \mathrm{Re}(e^{\pi i\sigma-2\pi ia}\zeta_L(s,\lambda,1-a) \overline{\zeta_L(s,1-\lambda,a)}) dt \Bigg\}
\end{align*}
We consider the first integral term in the above.
Let $ Z(u) = \int_1^u | \zeta_L(\sigma+it, a, \lambda) |^2 dt $, so
$ Z'(u) = |\zeta_L(\sigma+iu,a,\lambda)|^2$. Then, by applying integration by parts, the term $|\zeta_L(s,a)| t^{2\sigma-1}$, can be evaluated as follows:
\begin{align}
 & \int_{1}^{T} t^{2\sigma-1} |\zeta_L(\sigma+it,a,\lambda)|^2 dt \nonumber\\
 &= \bigg[ Z(t) t^{2\sigma-1}  \bigg]_1^T
    - (2 \sigma-1)\int_1^T Z(t) t^{2\sigma-2} dt  \nonumber\\
 %&= Z(T)T^{2\sigma-1}- (2 \sigma-1)\int_{1}^{T} Z(t) t^{2\sigma-2} dt \nonumber\\
 &= T^{2\sigma-1} \int_{1}^{T} |\zeta_L(\sigma+it,a,\lambda)|^2 dt 
    - (2 \sigma-1)\int_{1}^{T} Z(t) t^{2\sigma-2} dt\nonumber\\
 %&= \zeta_H(2\sigma,a)T^{2 \sigma}+\frac{(2\pi)^{2\sigma-1}}{2-2\sigma}\zeta(2-2\sigma)T +O(T^{\sigma} \log{T}) + O(T^{5\sigma/2-1}) \nonumber\\
 %&  \quad - (2 \sigma-1)\int_1^T \left\{\zeta_H(2\sigma,a)t+\frac{(2\pi)^{2\sigma-1}}{2-2\sigma}\zeta(2-2\sigma)t^{2-2\sigma}  
 % +O(t^{1-\sigma} \log{T}) + O(t^{\sigma/2}) \right\} t^{2\sigma-2} dt \nonumber\\
 &=\zeta_H(2\sigma,a)T^{2 \sigma}+\frac{(2\pi)^{2\sigma-1}}{2-2\sigma}\zeta_H(2-2\sigma,\lambda)T +O(T^{\sigma} \log{T}) \nonumber\\
 &  \quad - (2 \sigma-1)\int_1^T \left\{\zeta_H(2\sigma,a)t^{2\sigma-1}+\frac{(2\pi)^{2\sigma-1}}{2-2\sigma}\zeta_H(2-2\sigma,\lambda)  
  +O(t^{\sigma-1} \log{t}) \right\} dt
  \nonumber\\
 &= \frac{1}{2\sigma}\zeta_H(2\sigma,a)T^{2 \sigma}+(2\pi)^{2\sigma-1}\zeta_H(2-2\sigma,\lambda)T+O(T^{\sigma}\log{T})
 \nonumber
 %\label{zeta_L_t2sigma-1}
\end{align}
Thus
\begin{align}
 & \int_{1}^{T} t^{2\sigma-1} |\zeta_L(\sigma+it,1-\lambda,a)|^2 dt \nonumber\\
 &= \frac{1}{2\sigma}\zeta_H(2\sigma,1-\lambda)T^{2 \sigma}+(2\pi)^{2\sigma-1} \zeta_H(2-2\sigma,a)T+O(T^{\sigma}\log{T}).
 %\label{zeta_L_t2sigma-1}.
\end{align}
Then we have
\begin{align*}
 &\int_1^T|\zeta_L(1-\sigma+it,a,\lambda)|^2 dt \\
 &= (2\pi)^{1-2\sigma} \int_1^T t^{2\sigma-1} |\zeta_L(1-\sigma+it,1-\lambda,a)|^2 dt \\
 & \qquad \qquad + 2 (2\pi)^{1-2\sigma} \int_1^T t^{2\sigma-1}e^{-\pi t} \mathrm{Re} (e^{\pi i \sigma-2\pi ia}\zeta_L(s,\lambda,1-a) \overline{\zeta_L(s,1-\lambda,a)}) dt \\
 & \qquad \qquad + O(T^{2\sigma}e^{-2\pi T}) + O(T^{2\sigma-1}) \\
 &= (2\pi)^{1-2\sigma} 
  \left\{\frac{1}{2\sigma}\zeta_H(2\sigma,1-\lambda)T^{2 \sigma}+(2\pi)^{2\sigma-1}\zeta_H(2-2\sigma,a)T+O(T^{\sigma}\log{T}) \right\} \\
 & \qquad \qquad + 2 (2\pi)^{1-2\sigma} \int_1^T t^{2\sigma-1} e^{-\pi t} \mathrm{Re}(e^{\pi i \sigma-2\pi ia}\zeta_L(s,\lambda,1-a) \overline{\zeta_L(s,1-\lambda,a)}) dt \\
 & \qquad \qquad + O(T^{2\sigma}e^{-2\pi T}) + O(T^{2\sigma-1})\\
 &=\frac{(2\pi)^{1-2\sigma}}{2\sigma}\zeta_H(2\sigma,1-\lambda)T^{2\sigma}+\zeta_H(2-2\sigma,a)T+O(T^{\sigma}\log{T}).
\end{align*}
The cross term is bounded by
\begin{align*}
 &\int_1^T t^{2\sigma-1}e^{-\pi t}\,
\bigl|\zeta_L(s,\lambda,1-a)\bigr|\,
\bigl|\zeta_L(s,1-\lambda,a)\bigr|\,dt \\
& \qquad \qquad  \le
\frac12\int_1^T t^{2\sigma-1}e^{-\pi t}
\Bigl(|\zeta_L(s,\lambda,1-a)|^2+|\zeta_L(s,1-\lambda,a)|^2\Bigr)\,dt.
\end{align*}
Using a standard polynomial growth bound for the Lerch zeta function,
$|\zeta_L(\sigma+it,\alpha,\beta)|\ll (1+|t|)^A$ for some $A>0$,
the last integral is dominated by
$\int_1^\infty t^{2\sigma-1+2A}e^{-\pi t}\,dt<\infty$.
Hence this contribution is $O(1)$ as $T\to\infty$.

Finally,  replacing to $s$ by $ 1-s $, we have
\begin{align*}
 &\int_1^T|\zeta_L(\sigma+it,a,\lambda)|^2 dt \\
 &\qquad =\frac{(2\pi)^{2\sigma-1}}{2-2\sigma}\zeta_H(2-2\sigma,1-\lambda)T^{2-2\sigma}+\zeta_H(2\sigma,a)T+O(T^{1-\sigma}\log{T}).
\end{align*}
The case $\lambda=1$ can also be treated similarly by means of \eqref{hz_fq}, noting that the term $\zeta_H(2-2\sigma, 1-\lambda)$ is replaced by $\zeta(2-2\sigma)$.
% However, in the right-hand side above, the term $ \zeta_H(2\sigma,a)T $ is absorbed into $O(T^{3/2 - 5\sigma/2})$ when $0< \sigma \le 1/5$.
\end{proof}
\bigskip

\section{Auxiliary lemmas}
%Hence, their contributions to the integral are bounded, and the remaining integrals are $ O(1) $.
We also use the following lemmas.

\subsection{Mixed mean values for Hurwitz zeta functions}
\begin{lemma}[Montgomery--Vaughan theorem]\label{MonVau}
 Let $a_m \in \C$. There exists an absolute constant $c$ such that
 \begin{align}\label{MontVau_ineq}
     \left| \sum_{\substack{1\le m,n\le N \\ m\ne n}} \frac{a_m \bar{a}_n}{(m+a)^{\sigma}(n+a)^{\sigma} \log((m+a)/(n+a))} \right|
     \le c \sum_{m=1}^{N}\frac{m|a_m|^2}{(m+a)^{2\sigma}}.
 \end{align}
\end{lemma}

This inequality is a special case of the Montgomery--Vaughan theorem
\cite{MontgomeryVaughan1974}. A proof can be found in
Theorem 1.4.2 of \cite{Ramachandra1995}.
\medskip

\begin{lemma}\label{lem:hz_mixed_mean}
Let $s=\sigma+it$ with $r-1<\sigma<r$ and $0<a\le 1$.
Let $k$ and $l$ be non-negative integers such that $0\le k,l\le r-1$ and $(k,l)\ne (r-1,r-1)$.
If $k+l=2r-3$, then
\begin{align*}
 &\int_1^T \zeta_H(s-k,a) \overline{\zeta_H(s-l,a)} \, dt \\
 & \qquad = \zeta_H(2\sigma-2r+3,a)T +
 \begin{cases}
  O(1) & \text{if } r-1/2 < \sigma <r, \\[1ex]
  O\left(\log T\right) & \text{if } \sigma = r-1/2, \\[1ex]
  O(T^{2r-2\sigma-1}) & \text{if } r-1<\sigma<r-1/2.
 \end{cases}
\end{align*}
If $0 \le k+l \le 2r-4$, then
\begin{align*}
 \int_1^T \zeta_H(s-k,a) \overline{\zeta_H(s-l,a)} \, dt 
 = \zeta_H(2\sigma-k-l,a)T + O(1).
\end{align*}
The implied constants depend only on $\sigma$, $a$, $k$, and $l$.
\end{lemma}

\begin{proof}
First, assume that $0 \le k,l \le r-2$.
Since $r-1<\sigma<r$, we have $\sigma-k>1$ and $\sigma-l>1$. Hence
the Dirichlet series for $\zeta_H(s-k,a)$ and $\zeta_H(\bar s-l,a)$ converge absolutely, and we may write
\begin{align*}
 \zeta_H(s-k,a)\overline{\zeta_H(s-l,a)}
 &= \sum_{m=0}^{\infty}\frac{1}{(m+a)^{2\sigma-k-l}}
 + \sum_{\substack{m,n\ge 0\\ m\ne n}}
   \frac{1}{(m+a)^{\sigma-k}(n+a)^{\sigma-l}}
   \Bigl(\frac{n+a}{m+a}\Bigr)^{it}.
\end{align*}
Integrating term-by-term, we obtain
\begin{align*}
 &\int_{1}^{T}\zeta_H(s-k,a)\overline{\zeta_H(s-l,a)}\,dt \\
 &\quad = (T-1)\sum_{m=0}^{\infty}\frac{1}{(m+a)^{2\sigma-k-l}}
 + \sum_{\substack{m,n\ge 0\\ m\ne n}}
   \frac{1}{(m+a)^{\sigma-k}(n+a)^{\sigma-l}}
   \int_{1}^{T}\Bigl(\frac{n+a}{m+a}\Bigr)^{it}\,dt\\
 &\quad = \zeta_H(2\sigma-k-l,a)\,T + O(1),
\end{align*}
because 
$$\int_1^T \left(\frac{n+a}{m+a} \right)^{it}dt
= \frac{e^{iT\log{((n+a)/(m+a))}}-e^{i\log{((n+a)/(m+a))}}}{i\log{((n+a)/(m+a))}}$$
and the off-diagonal double series converges absolutely when $\sigma-k>1$ and $\sigma-l>1$.
This proves the second assertion in the range $0\le k,l\le r-2$.

\medskip
\noindent
Next, assume that $k=r-1$ or $l=r-1$.
By symmetry we may assume $k=r-1$ and $0\le l\le r-2$.
To avoid non-absolutely convergent series, we use the standard approximation \eqref{Hurwitz_finite} with $x=t$:
\begin{align*}
 \zeta_H(\sigma'+it,a)
 = \sum_{n\le t}\frac{1}{(n+a)^{\sigma'+it}}
   + \frac{t^{1-(\sigma'+it)}}{\sigma'+it-1}
   + O(t^{-\sigma'})
\qquad (t\ge 1),
\end{align*}
valid uniformly for $0<\sigma'\le 2$.
Applying this with $\sigma'=\sigma-r+1$ (and also with $\sigma'=\sigma-l$) gives, for $t\ge 1$,
\begin{align}
 \zeta_H(s-r+1,a)
 &= \sum_{n\le t}\frac{1}{(n+a)^{s-r+1}} + O(t^{r-1-\sigma}),
 \label{eq:approx_r-1}\\
 \zeta_H(s-l,a)
 &= \sum_{n\le t}\frac{1}{(n+a)^{s-l}} + O(t^{l-\sigma}).
 \label{eq:approx_l}
\end{align}
Indeed, since $|s-r|\asymp t$ and $|t^{r-s}|=t^{r-\sigma}$, we have
$t^{r-s}/(s-r)=O(t^{r-\sigma-1})$, and similarly for $l$.
Multiplying \eqref{eq:approx_r-1} and the conjugate of \eqref{eq:approx_l}, we obtain
\begin{align*}
 \zeta_H(s-r+1,a)\overline{\zeta_H(s-l,a)}
 &= \sum_{m,n\le t}\frac{1}{(m+a)^{s-r+1}(n+a)^{\bar s-l}}
 + O\!\left(t^{l-\sigma}\sum_{m\le t}\frac{1}{(m+a)^{\sigma-r+1}}\right) \\
 &\quad + O\!\left(t^{r-1-\sigma}\sum_{n\le t}\frac{1}{(n+a)^{\sigma-l}}\right)
 + O(t^{r+l-1-2\sigma}).
\end{align*}
Integrating from $1$ to $T$ and separating the diagonal part $m=n$ yields
\begin{align}
 \int_{1}^{T}\zeta_H(s-r+1,a)\overline{\zeta_H(s-l,a)}\,dt
 &= \zeta_H(2\sigma-r+1-l,a)\,T + E_{\mathrm{off}}(T) + E_{\mathrm{rem}}(T),
 \label{eq:main_decomp}
\end{align}
where $E_{\mathrm{off}}(T)$ is the off-diagonal contribution
\begin{align*}
 E_{\mathrm{off}}(T)
 := \sum_{\substack{m,n\le T\\ m\ne n}}
 \frac{e^{iT\log{((n+a)/(m+a))}}-e^{i\log{((n+a)/(m+a))}}}
 {(m+a)^{\sigma-r+1}(n+a)^{\sigma-l}\log{((n+a)/(m+a))}},
\end{align*}
and $E_{\mathrm{rem}}(T)$ is the contribution coming from the error terms
in \eqref{eq:approx_l}.
\noindent
Let $a_m=(m+a)^{(r-l-1)/2}$, then
\[
\frac{a_m\overline{a_n}}{(m+a)^{\sigma-(r-1+l)/2}(n+a)^{\sigma-(r-1+l)/2}}
=\frac{1}{(m+a)^{\sigma-r+1}(n+a)^{\sigma-l}}.
\]
Hence Lemma~\ref{MonVau} gives
\begin{align*}
 |E_{\mathrm{off}}(T)|
 &\ll \sum_{m\le T}\frac{m}{(m+a)^{2\sigma-(r-1)-l}}
 \ll \sum_{m\le T}\frac{1}{(m+a)^{2\sigma-(r-1)-l-1}}.
\end{align*}
If $l\le r-3$, then $2\sigma-(r-1)-l-1 \ge 2\sigma-2r+3>1$, so $|E_{\mathrm{off}}(T)|=O(1)$.
If $l=r-2$ (equivalently, $k+l=2r-3$), then
$2\sigma-(r-1)-l-1=2\sigma-2r+2$, and therefore
\[
|E_{\mathrm{off}}(T)|
=
\begin{cases}
O(1) & (r-1/2<\sigma<r),\\
O(\log T) & (\sigma=r-1/2),\\
O(T^{2r-2\sigma-1}) & (r-1<\sigma<r-1/2).
\end{cases}
\]
\smallskip
\noindent
Using the elementary bound
\[
\sum_{n\le t}\frac{1}{(n+a)^\alpha}\ll
\begin{cases}
1 & (\alpha>1),\\
\log t & (\alpha=1),\\
t^{1-\alpha} & (0<\alpha<1),
\end{cases}
\]
together with $\sigma-l>0$ and $\sigma-r+1\in(0,1)$, we obtain
\[
E_{\mathrm{rem}}(T)\ll
\int_1^T t^{l-\sigma}\sum_{m\le t}\frac{1}{(m+a)^{\sigma-r+1}}\,dt
+\int_1^T t^{r-1-\sigma}\sum_{n\le t}\frac{1}{(n+a)^{\sigma-l}}\,dt
+\int_1^T t^{r+l-1-2\sigma}\,dt.
\]
If $l\le r-3$, then $\sigma-l>1$ and the second integral is $O(1)$, while the first and third
integrals are also $O(1)$ because $l-\sigma+(r-\sigma)<-1$ and $r+l-1-2\sigma<-1$.
Hence $E_{\mathrm{rem}}(T)=O(1)$.
If $l=r-2$, then $\sigma-l=\sigma-r+2>1$ still holds, so the second integral is $O(1)$, and
the remaining two integrals are
\[
\ll \int_1^T t^{r-2-\sigma}\,t^{r-\sigma}\,dt + \int_1^T t^{2r-3-2\sigma}\,dt
\ll T^{2r-2\sigma-1}.
\]
Therefore, in all cases we have
\[
E_{\mathrm{rem}}(T)
=
\begin{cases}
O(1) & (l\le r-3),\\
O(T^{2r-2\sigma-1}) & (l=r-2).
\end{cases}
\]
Combining \eqref{eq:main_decomp} with the above estimates, we obtain the desired assertions
for $k=r-1$. The case $l=r-1$ is analogous, and the lemma follows.
\end{proof}

\bigskip

\subsection{Truncation formulas and comparison principles}
\begin{lemma}[Theorem 1.2 of \cite{MiyagawaMurahara2025}]\label{MiyaMura_Thm1.2}
Let $ r-1 < \sigma_1 < \sigma_2 $, $x \geq 1$ and $C>1$. If $s= \sigma+it \in \C $ with $ \sigma_1 < \sigma < \sigma_2 $ and $ |t| \leq 2 \pi x / C $, then
\begin{align*} 
 \zeta_r(s,a,\mathbf{w})
 &=\sum_{0\le m_1\le x}
  \cdots 
  \sum_{0\le m_r\le x}
  \frac{1}{(a+\mathbf{m}\cdot\mathbf{w})^{s}}\\
 &\quad -
 \sum_{\substack{E\subseteq\{w_1,\dots,w_r\}\\E\ne\emptyset}}
 (-1)^{\#E}
 \frac{(a+x\sum_{ e\in E }e)^{r-s}}{(s-1)\cdots (s-r) w_1\cdots w_{r}}
 +O(x^{r-1-\sigma})
\end{align*}
as $x\rightarrow \infty$.
\end{lemma}

\begin{corollary}\label{cor:trunc_t}
Let $r-1<\sigma_1<\sigma_2<r$. Then, uniformly for $\sigma_1<\sigma<\sigma_2$ and $t\ge1$,
\[
\zeta_r(s,a,\mathbf w)
=
\sum_{0\le m_1,\dots,m_r\le t}
\frac{1}{(a+\mathbf m\cdot\mathbf w)^s}
+O(t^{r-1-\sigma}).
\]
\end{corollary}
\begin{proof}
Set $C=2\pi$ and $x=t$ in Lemma~\ref{MiyaMura_Thm1.2}. Then $|t|\le 2\pi x/C$ holds trivially.
Moreover, for each nonempty subset $E\subseteq\{w_1,\dots,w_r\}$, we have
\[
\frac{(a+t\sum_{e\in E} e)^{r-s}}{(s-1)\cdots(s-r)w_1\cdots w_r}
\ll t^{r-\sigma}|t|^{-r}
\ll t^{-\sigma},
\]
hence the whole secondary term is $O(t^{-\sigma})$, which is absorbed by
$O(t^{r-1-\sigma})$ since $r\ge1$. This proves the corollary.
\end{proof}

\begin{lemma}\label{lem:upper_comparison}
Let $r-1<\sigma_1<\sigma_2<r$. Then, uniformly for
$s=\sigma+it$ with $\sigma_1<\sigma<\sigma_2$ and $t\ge1$,
\[
|\zeta_r(s,a,\mathbf w)|
\ll
\sum_{0\le m_1,\dots,m_r\le t}
\frac{1}{(a+\mathbf m\cdot\mathbf 1)^\sigma}
+ t^{r-1-\sigma}
\]
as $t \to \infty$, where the implied constant depends only on $r,\sigma_1,\sigma_2,a$, and $\mathbf w$.
\end{lemma}
\begin{proof}
By Corollary~\ref{cor:trunc_t}, we have
\[
\zeta_r(s,a,\mathbf w)
=
\sum_{0\le m_1,\dots,m_r\le t}
\frac{1}{(a+\mathbf m\cdot\mathbf w)^s}
+O(t^{r-1-\sigma}),
\]
uniformly for $\sigma_1<\sigma<\sigma_2$.
Hence, by the triangle inequality,
\begin{align*}
|\zeta_r(s,a,\mathbf w)|
&\le
\sum_{0\le m_1,\dots,m_r\le t}
\frac{1}{(a+\mathbf m\cdot\mathbf w)^\sigma}
+O(t^{r-1-\sigma}).
\end{align*}
Let $ w_{\min}=\min\{w_1, \dots,w_r \} $ and put $a'=a/w_{\min}$. Then 
$ a+\mathbf m\cdot\mathbf w \ge w_{\min}(a’+\mathbf m\cdot\mathbf 1) $ and hence
$ (a+\mathbf m\cdot\mathbf w)^{-\sigma}
\le
w_{\min}^{-\sigma}(a'+\mathbf m\cdot\mathbf 1)^{-\sigma}
\ll
(a+\mathbf m\cdot\mathbf 1)^{-\sigma}
$.
Since $a'>0$ is fixed, we have
$ a'+x\asymp a+x \ (x\ge0) $, uniformly for $\sigma_1<\sigma<\sigma_2$. Therefore,
 \[
 |\zeta_r(s,a,\mathbf w)|
 \ll
 \sum_{0\le m_1,\dots,m_r\le t}
 \frac{1}{(a+\mathbf m\cdot\mathbf 1)^\sigma}
 +t^{r-1-\sigma},
 \]
as claimed.
\end{proof}

\bigskip

\subsection{Diagonal terms in the rank-one case}
\begin{lemma}\label{lem:A_n_asymp}
Let $p_1,\dots,p_r\in\Z_{>0}$ and define
\[
A(n)=\#\left\{\mathbf m\in\Z_{\ge0}^r:
p_1m_1+\cdots+p_rm_r=n\right\}.
\]
Then, as $n\to\infty$,
\[
A(n)=\frac{n^{r-1}}{(r-1)!\,p_1\cdots p_r}
+O(n^{r-2}).
\]
Consequently,
\[
A(n)^2=
\frac{n^{2r-2}}{\{(r-1)!\}^2(p_1\cdots p_r)^2}
+O(n^{2r-3}).
\]
\end{lemma}
\begin{proof}[Sketch of proof]
Consider the polytope
\[
\mathcal P_n:=\Bigl\{\mathbf x\in\R_{\ge0}^r:\ p_1x_1+\cdots+p_rx_r=n\Bigr\}.
\]
Then $A(n)$ equals the number of lattice points on $\mathcal P_n$.
After the scaling $x_j=n y_j$, we obtain the simplex
\[
\mathcal P_1=\Bigl\{\mathbf y\in\R_{\ge0}^r:\ p_1y_1+\cdots+p_ry_r=1\Bigr\},
\]
whose $(r-1)$-dimensional volume is
\[
\operatorname{Vol}_{r-1}(\mathcal P_1)=\frac{1}{(r-1)!\,p_1\cdots p_r}.
\]
(Standard: in coordinates $u_j=p_jy_j$ the hyperplane becomes $u_1+\cdots+u_r=1$, and the Jacobian is
$(p_1\cdots p_r)^{-1}$.)
Hence the main term is $\operatorname{Vol}_{r-1}(\mathcal P_1)\,n^{r-1}$.
A standard lattice-point estimate for dilates of rational polytopes (e.g. Ehrhart theory, or a volume
argument with boundary contribution) yields the error term $O(n^{r-2})$, giving
\[
A(n)=\frac{n^{r-1}}{(r-1)!\,p_1\cdots p_r}+O(n^{r-2}).
\]
Squaring this asymptotic gives the stated formula for $A(n)^2$.
\end{proof}

\medskip

\begin{lemma}\label{lem:tilde_zeta_pole}
Assume $d=1$ and write $w_j=\lambda p_j$ $(1\le j\le r)$.
Then the diagonal series
\[
\tilde{\zeta}_r(s,a,\mathbf w)
=
\sum_{n=0}^{\infty}
\frac{A(n)^2}{(a+\lambda n)^{2s}}
\]
admits a meromorphic continuation to $\Re(s)>r-1$
and has a simple pole at $s=r-1/2$ with residue
\[
\frac{\lambda^{1-2r}}{2\{(r-1)!\}^2(p_1\cdots p_r)^2}.
\]
\end{lemma}
\begin{proof}
By Lemma~\ref{lem:A_n_asymp}, write
\[
A(n)^2= \frac{1}{\{(r-1)!\}^2(p_1\cdots p_r)^2} n^{2r-2}+O(n^{2r-3}),
\]
Put
\[
\tilde{\zeta}_r(s,a,\mathbf w)
=\frac{A(0)^2}{a^{2s}}+\sum_{n\ge1}\frac{A(n)^2}{(a+\lambda n)^{2s}}.
\]
Split the sum for $n\ge1$ as
\begin{align*}
\sum_{n\ge1}\frac{A(n)^2}{(a+\lambda n)^{2s}}
&=\frac{1}{\{(r-1)!\}^2(p_1\cdots p_r)^2}\sum_{n\ge1}\frac{n^{2r-2}}{(a+\lambda n)^{2s}} \\
& \qquad +\sum_{n\ge1}\frac{1}{(a+\lambda n)^{2s}} 
 \left\{ A(n)^2-\frac{n^{2r-2}}{\{(r-1)!\}^2(p_1\cdots p_r)^2}  \right\} \\
&= \frac{1}{\{(r-1)!\}^2(p_1\cdots p_r)^2}\sum_{n\ge1}\frac{n^{2r-2}}{(a+\lambda n)^{2s}} 
 + O \left( \sum_{n \ge 1} \frac{n^{2r-3}}{(a+\lambda n)^{2s}} \right).
\end{align*}
The second series converges absolutely for $\Re(s)>r-1$, since the numerator is $O(n^{2r-3})$,
and hence defines a holomorphic function on $\Re(s)>r-1$.
For the first series, write
\[
(a+\lambda n)^{-2s}=\lambda^{-2s}n^{-2s}\Bigl(1+\frac{a}{\lambda n}\Bigr)^{-2s}.
\]
Then
\[
\sum_{n\ge1}\frac{n^{2r-2}}{(a+\lambda n)^{2s}}
=
\sum_{n\ge1}\frac{\lambda^{-2s}}{n^{2s-2r+2}}
+\sum_{n\ge1}\frac{\lambda^{-2s}}{n^{2s-2r+2}}
\Bigl\{\Bigl(1+\frac{a}{\lambda n}\Bigr)^{-2s}-1\Big\}.
\]
The bracket is $O(1/n)$ uniformly on vertical strips, so the second sum converges absolutely for
$\Re(s)>r-1$ and is holomorphic there. Therefore we obtain, for $\Re(s)>r-1$,
\[
\tilde{\zeta}_r(s,a,\mathbf w)
=
\frac{\lambda^{-2s}}{\{(r-1)!\}^2(p_1\cdots p_r)^2}\zeta(2s-2r+2)+H(s),
\]
where $H(s)$ is holomorphic on $\Re(s)>r-1$.
Since $\zeta(u)$ has a simple pole at $u=1$ with residue $1$, the only pole in this region occurs
when $2s-2r+2=1$, i.e. $s=r-1/2$. The residue equals
\[
\left[\frac{\lambda^{-2s}}{\{(r-1)!\}^2(p_1\cdots p_r)^2}\cdot \frac{1}{2}\right]_{s=r-1/2}
=\frac{\lambda^{1-2r}}{2\{(r-1)!\}^2(p_1\cdots p_r)^2}.
\]
This proves the lemma.
\end{proof}
\bigskip

\begin{lemma}\label{lem:log_asymp}
Under the assumptions of Lemma~\ref{lem:tilde_zeta_pole},
\[
\sum_{n\le X}
\frac{A(n)^2}{(a+\lambda n)^{2r-1}}
=
\frac{\lambda^{1-2r}}{\{(r-1)!\}^2(p_1\cdots p_r)^2}
\log X
+O(1).
\]
\end{lemma}
\begin{proof}
By Lemma~\ref{lem:A_n_asymp},
\[
A(n)^2=\frac{n^{2r-2}}{\{(r-1)!\}^2(p_1\cdots p_r)^2}+O(n^{2r-3}).
\]
Moreover,
$
(a+\lambda n)^{-(2r-1)}=\lambda^{1-2r}n^{-(2r-1)}(1+O(1/n)),
$
so
\[
\frac{A(n)^2}{(a+\lambda n)^{2r-1}}
=
\frac{\lambda^{1-2r}}{\{(r-1)!\}^2(p_1\cdots p_r)^2}\frac{1}{n}+O\Bigl(\frac{1}{n^2}\Bigr).
\]
Summing over $n\le X$, we obtain this lemma.
\end{proof}

\begin{lemma}\label{lem:rank1_hurwitz_decomp_explicit}
Assume that $d=1$, and write $w_j=\lambda p_j$ $(1\le j\le r)$,
where $\lambda>0$ and $p_1,\dots,p_r\in\mathbb Z_{>0}$ with
$\gcd(p_1,\dots,p_r)=1$.
Let $q=\operatorname{lcm}(p_1,\dots,p_r)$.
Then there exist constants $c_{k,\nu}$ $(0\le k\le r-1,\ 0\le \nu\le q-1)$
and positive real numbers $\alpha_\nu$ such that
\[
\zeta_r(s,a,\mathbf w)
=
(\lambda q)^{-s}
\sum_{\nu=0}^{q-1}\sum_{k=0}^{r-1}
c_{k,\nu}\,\zeta_H(s-k,\alpha_\nu),
\]
where $\zeta_H(s,a)$ denotes the Hurwitz zeta function.
\end{lemma}
\begin{proof}
Since $d=1$, we may write $w_j=\lambda p_j$ with
$\lambda>0$ and $p_j\in\mathbb Z_{>0}$.
Then
\[
\zeta_r(s,a,\mathbf w)
=
\sum_{m_1,\dots,m_r \ge 0 }
\frac{1}{(a+\lambda(p_1m_1+\cdots+p_rm_r))^s}.
\]
Grouping terms according to the value $ n=p_1m_1+\cdots+p_rm_r$,
we obtain
\[
\zeta_r(s,a,\mathbf w)
=
\sum_{n=0}^\infty \frac{A(n)}{(a+\lambda n)^s},
\]
where
\[
A(n)
=
\#\{\mathbf m\in\mathbb Z_{\ge0}^r:
p_1m_1+\cdots+p_rm_r=n\}.
\]
It is well known that $A(n)$ is a quasi-polynomial of degree $r-1$
with period $q=\operatorname{lcm}(p_1,\dots,p_r)$.
Hence we may write
\[
A(n)
=
\sum_{k=0}^{r-1} c_{k,\nu}\,n^k
\quad (n\equiv \nu \bmod q),
\]
for suitable constants $c_{k,\nu}$. Therefore
\[
\zeta_r(s,a,\mathbf w)
=
\sum_{\nu=0}^{q-1}
\sum_{\substack{n\ge0\\ n\equiv \nu\ (\mathrm{mod}\ q)}}
\frac{A(n)}{(a+\lambda n)^s}.
\]
Writing $n=q m+\nu$, we obtain
\[
\zeta_r(s,a,\mathbf w)
=
\sum_{\nu=0}^{q-1}
\sum_{m=0}^\infty
\frac{A(qm+\nu)}{(a+\lambda(qm+\nu))^s}.
\]
Using the polynomial expression of $A(qm+\nu)$ and expanding
$(qm+\nu)^k$ as a polynomial in $m$, we can express the above as
a finite linear combination of series of the form
\[
\sum_{m=0}^\infty (m+\alpha_\nu)^{-s+k}
=
\zeta_H(s-k,\alpha_\nu),
\]
where $\alpha_\nu=(a+\lambda \nu)/(\lambda q)$.
Thus we obtain the desired decomposition.
\end{proof}

\bigskip

\begin{lemma}\label{lem:rank1_top_coeff_explicit}
Assume that $d=1$, and write $w_j=\lambda p_j$ $(1\le j\le r)$,
where $\lambda>0$ and $p_1,\dots,p_r\in\mathbb Z_{>0}$ with
$\gcd(p_1,\dots,p_r)=1$.
Let $q=\operatorname{lcm}(p_1,\dots,p_r)$.
In the decomposition
\[
\zeta_r(s,a,\mathbf w)
=
(\lambda q)^{-s}
\sum_{\nu=0}^{q-1}\sum_{k=0}^{r-1}
c_{k,\nu}\,\zeta_H(s-k,\alpha_\nu),
\]
the coefficients corresponding to the top shift $k=r-1$ satisfy
\[
\sum_{\nu=0}^{q-1} c_{r-1,\nu}
=
\frac{q^{\,r-1}}{(r-1)!\,p_1\cdots p_r}.
\]
Consequently, the contribution of the top-shift terms to the mean square value is
 \[
 \frac{\lambda^{1-2r}}{\{(r-1)!\}^2(p_1\cdots p_r)^2}.
 \]
\end{lemma}
\begin{proof}
Recall that
\[
\zeta_r(s,a,\mathbf w)
=
\sum_{n=0}^\infty \frac{A(n)}{(a+\lambda n)^s},
\]
where
\[
A(n)
=
\#\{\mathbf m\in\mathbb Z_{\ge0}^r:
p_1m_1+\cdots+p_rm_r=n\}.
\]
By Lemma~\ref{lem:A_n_asymp}, we have
\[
A(n)
=
\frac{n^{r-1}}{(r-1)!\,p_1\cdots p_r}
+O(n^{r-2}).
\]
For each residue class $\nu \pmod q$, write $n=qm+\nu$. Then
\[
A(qm+\nu)
=
\frac{(qm+\nu)^{r-1}}{(r-1)!\,p_1\cdots p_r}+O(m^{r-2})
=
\frac{q^{r-1}m^{r-1}}{(r-1)!\,p_1\cdots p_r}+O(m^{r-2})
\]
as $m \to \infty$.
On the other hand, in the proof of Lemma~\ref{lem:rank1_hurwitz_decomp_explicit},
the coefficient $c_{k,\nu}$ arises from the polynomial expansion of $A(qm+\nu)$
as a polynomial in $m$.
Since $A(n)$ is a quasi-polynomial of degree $r-1$ with period $q$,
its leading coefficient is independent of the residue class $\nu$.
Therefore the coefficient of the top-degree term is
\[
c_{r-1,\nu}
=
\frac{q^{r-1}}{(r-1)!\,p_1\cdots p_r}.
\]
Hence
\[
\sum_{\nu=0}^{q-1} c_{r-1,\nu}^2
=
q\left(\frac{q^{r-1}}{(r-1)!\,p_1\cdots p_r}\right)^2
=
\frac{q^{2r-1}}{\{(r-1)!\}^2(p_1\cdots p_r)^2}.
\]
This proves the lemma.
\end{proof}

\section{Proofs of the Main Theorems}

\medskip
\noindent
\textbf{Proof of Theorem \ref{zeta_r(1/2,a)}}\\
Using \eqref{HMZ-HZ}, the mean square values of \( \zeta_r(\sigma + it, a, \mathbf{1}) \) are given by
\begin{align}
 \int_1^T |\zeta_r(s,a,\mathbf{1})|^2 dt 
 &=\sum_{k=0}^{r-1} p_{r,k}(a)^2 \int_1^T |\zeta_H(s-k, a)|^2dt   \nonumber \\
 & \quad + 2 \sum_{0\le k<l\le r-1} p_{r,k}(a)p_{r,l}(a)\cdot 
    \mathrm{Re}\left( \int_1^T \zeta_H(s-k,a) \overline{\zeta_H(s-l,a)} dt \right)
    \nonumber\\
 &=\frac{1}{((r-1)!)^2}\int_1^T |\zeta_H(\sigma-r+1,a)|^2dt + \sum_{k=0}^{r-2} p_{r,k}(a)^2 \int_1^T |\zeta_H(s-k, a)|^2dt \nonumber\\
 & \quad + 2 \sum_{0\le k<l\le r-1} p_{r,k}(a)p_{r,l}(a)\cdot 
    \mathrm{Re}\left( \int_1^T \zeta_H(s-k,a) \overline{\zeta_H(s-l,a)} dt \right).
    \label{int_zeta_r(s,a,1)}
\end{align}
By \eqref{G_and_L_2} and Proposition \ref{Hurwitz_0<siga<1/2}, we have 
% \begin{align*}
%     &\frac{1}{((r-1)!)^2} \int_1^T|\zeta_H(\sigma-r+1+it,a)|^2dt \\
%     &= \frac{1}{((r-1)!)^2}
%       \Bigg\{ \zeta_H(2\sigma-2r+2,a)T + \frac{(2\pi)^{2\sigma-2r+1}}{2r-2\sigma}\zeta(2r-2\sigma)T^{2r-2\sigma} \\
%     & \qquad \qquad \qquad \qquad + O(T^{r-\sigma}\log{T}) + O(T^{(\sigma-r+1)/2}) \Bigg\}
% \end{align*}
\begin{align*}
 & \int_1^T|\zeta_H(\sigma-r+1+it,a)|^2dt \\
 &= \begin{cases}
     \zeta_H(2\sigma-2r+2,a)T + \frac{(2\pi)^{2\sigma-2r+1}}{2r-2\sigma}\zeta(2r-2\sigma)T^{2r-2\sigma} \\
    \qquad \qquad \qquad \qquad \qquad \qquad \qquad \qquad + O(T^{r-\sigma}\log{T})  & \text{if } \sigma>r-1/2, \\
      T\log{T} + \left(\gamma(a) + \gamma -1-\log{2\pi} \right)T + O(T^{1/2} \log{T}) & \text{if } \sigma=r-1/2, \\
    \frac{(2\pi)^{2\sigma-2r+1}}{2r-2\sigma} \zeta(2r-2\sigma) T^{2r-2\sigma} +\zeta_H(2\sigma-2r+2,a)T\\
     \qquad \qquad \qquad \qquad \qquad \qquad \qquad \qquad   +O(T^{r-\sigma}\log{T}) & \text{if }  r-1<\sigma< r-1/2.
    \end{cases} 
\end{align*}
This contribution corresponds to the case $k=r-1$ in \eqref{int_zeta_r(s,a,1)} and yields the main term of that expression.
The second term of \eqref{int_zeta_r(s,a,1)} is estimated by using 
\begin{align*}
  \int_1^T |\zeta_H(s-k, a)|^2dt = \zeta_H(2\sigma-2k,a)T+O(1)
\end{align*}
for $0 \le k \le r-2$.
The third term of \eqref{int_zeta_r(s,a,1)} is estimated by
\begin{align*}
  &2\sum_{0\le k<l\le r-1} p_{r,k}(a)p_{r,l}(a)\cdot 
    \mathrm{Re}\left( \int_1^T \zeta_H(s-k,a) \overline{\zeta_H(s-l,a)} dt \right) \\
  &=2 p_{r,r-1}(a)p_{r,r-2}(a)\zeta_H(2\sigma-2r+3,a)T +
 \begin{cases}
  O(1) & \text{if } r-1/2 < \sigma <r, \\[1ex]
  O\left(\log T\right) & \text{if } \sigma = r-1/2, \\[1ex]
  O(T^{2r-2\sigma-1}) & \text{if } r-1<\sigma<r-1/2,
 \end{cases} \\
  & \quad + 2\sum_{\substack{0\le k<l\le r-1\\(k,l)\ne(r-2,r-1)}} p_{r,k}(a)p_{r,l}(a) \zeta_H(2\sigma-k-l,a) T + O(1) \\
  &= 2\sum_{0 \le k<l \le r-1}p_{r,k}(a)p_{r,l}(a) \zeta_H(2\sigma-k-l,a)T + 
 \begin{cases}
  O(1) & \text{if } r-1/2 < \sigma <r, \\[1ex]
  O\left(\log T\right) & \text{if } \sigma = r-1/2, \\[1ex]
  O(T^{2r-2\sigma-1}) & \text{if } r-1<\sigma<r-1/2.
 \end{cases}
\end{align*}
This completes the proof of Theorem \ref{zeta_r(1/2,a)}.
\qed

\bigskip

\bigskip

% \begin{remark}[On possible extensions beyond $\sigma>r-1$]
% The arguments in this paper rely essentially on finite-sum approximations
% of the Barnes multiple zeta function truncated at height $t$ (Lemma \ref{MiyaMura_Thm1.2}),
% which are valid only in the half-plane $\sigma>r-1$.
% This restriction prevents us from extending the present proofs
% to the full range of $\sigma$.

% Nevertheless, heuristic considerations based on the structure of the
% level sets $\{\mathbf m\in\mathbb Z_{\ge0}^r : \mathbf m\cdot\mathbf w = u\}$
% suggest that the order of the mean square value of
% $\zeta_r(\sigma+it,a,\mathbf w)$ should depend on
% the $\mathbb Q$-rank $d=\dim_{\mathbb Q}\langle w_1,\dots,w_r\rangle$,
% even outside the region $\sigma>r-1$.

% Establishing such results in a rigorous manner would require
% new ideas beyond the finite-sum approximation method employed here,
% and we leave this problem for future investigation.
% \end{remark}

\noindent
\textbf{Proof of Theorem \ref{thm:Barnes_bounds}}\\
\noindent{(i) The case $\sigma>r$.}
Since the defining Dirichlet series
\[
\zeta_r(s,a,\mathbf w)
=
\sum_{m_1,\dots,m_r \ge 0}
\frac{1}{(a+\mathbf m\cdot\mathbf w)^s}
\]
converges absolutely for $\sigma>r$, we obtain the trivial bound $ \zeta_r(\sigma+it,a,\mathbf w)\ll 1 $.

\medskip
\noindent{(ii) The case $\sigma=r$.}
By Lemma~\ref{lem:upper_comparison}, we have
\[
|\zeta_r(r+it,a,\mathbf w)|
\ll
\sum_{0\le m_1,\dots,m_r\le t}
\frac{1}{(a+\mathbf m\cdot\mathbf 1)^r}
+1.
\]
Writing $N=m_1+\cdots+m_r$, we obtain
\[
\sum_{0\le m_1,\dots,m_r\le t}
\frac{1}{(a+\mathbf m\cdot\mathbf 1)^r}
\ll
\sum_{N\le rt}\frac{N^{r-1}}{(a+N)^r}
\ll
\sum_{N\le rt}\frac{1}{N+1}
\ll \log t.
\]
Hence
\[
\zeta_r(r+it,a,\mathbf w)\ll \log t.
\]

\medskip
\noindent{(iii) The case $r-1<\sigma < r$.}
Again by Lemma~\ref{lem:upper_comparison}, we have
\[
|\zeta_r(\sigma+it,a,\mathbf w)|
\ll
\sum_{0\le m_1,\dots,m_r\le t}
\frac{1}{(a+\mathbf m\cdot\mathbf 1)^\sigma}
+t^{r-1-\sigma}.
\]
Writing $N=m_1+\cdots+m_r$, we obtain
\[
\sum_{0\le m_1,\dots,m_r\le t}
\frac{1}{(a+\mathbf m\cdot\mathbf 1)^\sigma}
\ll
\sum_{N\le rt}\frac{N^{r-1}}{(a+N)^\sigma}.
\]
Since $r-1<\sigma<r$, we have $-1<r-1-\sigma<0$, and hence
\[
\sum_{N\le rt}\frac{N^{r-1}}{(a+N)^\sigma}
\ll
\sum_{N\le rt}(N+1)^{r-1-\sigma}
\ll t^{r-\sigma}.
\]
Therefore
\[
\zeta_r(\sigma+it,a,\mathbf w)\ll t^{r-\sigma}.
\]
Since $t^{r-1-\sigma}=o(t^{r-\sigma})$, the error term is absorbed.
This completes the proof.
\qed

\bigskip

\noindent
\textbf{Proof of Theorem~\ref{main4_revised_v3}.}
\medskip

(I) The case $d=1$. 
Then there exist $\lambda>0$ and
$p_1,\dots,p_r\in\mathbb Z_{>0}$ such that
\[
w_j=\lambda p_j \qquad (1\le j\le r),
\]
with $\gcd(p_1,\dots,p_r)=1$.
\smallskip
(i) The case $\sigma=r-1/2$.
By Lemma~\ref{lem:rank1_hurwitz_decomp_explicit}, we may write
\[
\zeta_r(s,a,\mathbf w)
=
(\lambda q)^{-s}
\sum_{\nu=0}^{q-1}\sum_{k=0}^{r-1}
c_{k,\nu}\,\zeta_H(s-k,\alpha_\nu),
\]
where $q=\operatorname{lcm}(p_1,\dots,p_r)$ and $\alpha_\nu>0$.
Moreover, by Lemma~\ref{lem:rank1_top_coeff_explicit}, the terms with shift
$k=r-1$ occur with the explicit top coefficient giving the constant
\[
\frac{\lambda^{1-2r}}{\{(r-1)!\}^2(p_1\cdots p_r)^2}
\]
in the mean square formula.
Expanding
\[
\int_1^T |\zeta_r(r-1/2+it,a,\mathbf w)|^2\,dt
\]
by means of the above finite decomposition, we reduce the problem to a finite
linear combination of mean square and mixed mean values of Hurwitz zeta functions.
The self-products of the top-shift terms $\zeta_H(s-r+1,\alpha_\nu)$ contribute
the main term of order $T\log T$, while all remaining terms contribute at most $O(T)$.
Using Lemma~\ref{lem:rank1_top_coeff_explicit} to identify the coefficient of the
top-shift part, we obtain
\[
\int_1^T \left|\zeta_r\left(r-\frac12+it,a,\mathbf w\right)\right|^2\,dt
=
\frac{\lambda^{1-2r}}{\{(r-1)!\}^2(p_1\cdots p_r)^2}\,T\log T+O(T).
\]
(ii) The case $r-1<\sigma<r-1/2$.
Again by Lemma~\ref{lem:rank1_hurwitz_decomp_explicit},
\[
\zeta_r(s,a,\mathbf w)
=
(\lambda q)^{-s}
\sum_{\nu=0}^{q-1}\sum_{k=0}^{r-1}
c_{k,\nu}\,\zeta_H(s-k,\alpha_\nu).
\]
Applying Lemma~\ref{lem:hz_mixed_mean} to the mixed mean values of the Hurwitz
zeta functions appearing in this finite sum, we see that the self-products of the
top-shift terms $\zeta_H(s-r+1,\alpha_\nu)$ contribute terms of order
$T^{2r-2\sigma}$, whereas all terms involving at least one lower shift
$0\le k\le r-2$ are of smaller order. Since $r-1<\sigma<r-1/2$, we have
$2r-2\sigma>1$, and hence the top-shift contribution dominates.
Therefore
\[
\int_1^T |\zeta_r(\sigma+it,a,\mathbf w)|^2\,dt
\asymp T^{2r-2\sigma}.
\]

\medskip

(II) The case $2\le d\le r$.
In this case, the asserted upper bounds follow immediately from
Theorem~\ref{thm:Barnes_bounds} by squaring and integrating.
Indeed, if $\sigma>r$, then
$
\zeta_r(\sigma+it,a,\mathbf w)\ll 1,
$
hence
\begin{align*}\label{llT}
    \int_1^T |\zeta_r(\sigma+it,a,\mathbf w)|^2\,dt \ll T.
\end{align*}
If $\sigma=r$, then
$
\zeta_r(r+it,a,\mathbf w)\ll \log t,
$
and therefore
\[
\int_1^T |\zeta_r(r+it,a,\mathbf w)|^2\,dt
\ll
\int_1^T (\log t)^2\,dt
\ll
T(\log T)^2.
\]
Finally, if $r-1<\sigma<r$, then
$
\zeta_r(\sigma+it,a,\mathbf w)\ll t^{\,r-\sigma},
$
so that
\[
\int_1^T |\zeta_r(\sigma+it,a,\mathbf w)|^2\,dt
\ll
\int_1^T t^{2r-2\sigma}\,dt
\ll
T^{2r-2\sigma+1}.
\]
This proves the theorem.
\qed

\bigskip
\noindent
\textbf{Remark.}\\
In the case $2 \le d \le r$, we use the upper bound
in Theorem~\ref{thm:Barnes_bounds}. 
As an illustration, suppose that $w_1,\dots,w_r$ are algebraic numbers.
TThen, by Schmidt's Subspace Theorem, one has the following property:
for any $\varepsilon>0$, all integer vectors
$\mathbf{k}\in\mathbb{Z}^r$ satisfying
\[
0<|(\mathbf{m}-\mathbf{n})\cdot\mathbf{w}|
< \|\mathbf{m}-\mathbf{n}\|^{-d+1-\varepsilon}
\]
lie in finitely many proper subspaces of $\mathbb{Q}^r$.
Using this Diophantine estimate, one may expect that
\[
\int_1^{T} |\zeta_r(\sigma+it, a, \mathbf{w})|^2 dt
\ll T^{2r-2\sigma+d+\varepsilon}
\]
for any $\varepsilon>0$.
Since $d \ge 2$, this bound is worse than the trivial estimate
\[
\int_1^T |\zeta_r(\sigma+it,a,\mathbf w)|^2 dt \ll T^{2r-2\sigma+1}.
\]
This shows that the method based on the decomposition
into diagonal and off-diagonal terms is not effective
in the higher-rank case.
To obtain sharper results, a different approach would be required.
We leave this as a problem for future investigation.

\bigskip

\section*{Acknowledgments}
At the Kagawa Seminar in February 2026, Professor Masahiro Mine provided extremely valuable insights that contributed significantly to the further development of this paper. I would like to take this opportunity to express my deepest gratitude.

\bigskip

%%%%%%%%%%%%%%%%%%%%%%%%%%%%%%
\bibliographystyle{amsalpha}
\bibliography{References} 

\end{document}